\documentclass[11pt,reqno]{amsart}
\usepackage{xcolor}
\usepackage{amssymb, amsmath, amsthm,amscd,mathabx}
\usepackage[colorlinks,plainpages,backref,urlcolor=blue]{hyperref}
\hypersetup{colorlinks=true, citecolor=magenta}
\usepackage[alphabetic,lite]{amsrefs}
\usepackage{mathtools}
\usepackage{verbatim}
\usepackage{amscd}
\usepackage{tikz}
\usetikzlibrary{cd,calc,arrows}
\usepackage{mathrsfs}
\usepackage{upgreek}
\usepackage{cases}
\usepackage{comment}
\usepackage{colonequals}
\usepackage[shortlabels]{enumitem}
\usepackage[marginpar=2.4cm]{geometry}

\newcommand{\defi}[1]{{\emph{#1}}}
\newcommand{\bl}[1]{{\bf{#1}}}

\newcommand{\bw}{\bigwedge}

\renewcommand{\ll}{\lambda}

\newcommand{\oo}{\otimes}

\newcommand{\M}{\mathcal{M}}

\newcommand{\C}{\mathbb{C}}

\newcommand{\A}{\mathcal{A}}
\newcommand{\BB}{\mathcal{B}}
\newcommand{\NN}{\mathcal{N}}
\newcommand{\RR}{\mathcal{R}}
\newcommand{\XX}{\mathcal{X}}

\newcommand{\Sym}{\operatorname{Sym}}

\newcommand{\rank}{\operatorname{rank}}

\newcommand{\spn}{\operatorname{span}}
\newcommand{\Cent}{\operatorname{Cent}}

\newcommand{\bb}[1]{\mathbb{#1}}

\newcommand{\mc}[1]{\mathcal{#1}}

\newcommand{\op}[1]{\operatorname{#1}}

\def\lra{\longrightarrow}

\newcommand{\bwedge}{\mbox{\normalsize $\bigwedge$}}

\def\dot{\mathchar"013A}  
\newcommand{\hdot}{{\raise1pt\hbox to0.35em{\!\Huge $\dot$}}} 

\newtheorem{theorem}{Theorem}[section]
\newtheorem{lemma}[theorem]{Lemma}

\newtheorem{proposition}[theorem]{Proposition}
\newtheorem{corollary}[theorem]{Corollary}
\newtheorem*{theorem*}{Theorem}
\newtheorem*{problem*}{Problem}
\newtheorem*{claim*}{Claim}

\theoremstyle{definition}
\newtheorem{definition}[theorem]{Definition}
\newtheorem*{definition*}{Definition}

\newtheorem*{remark*}{Remark}
\newtheorem*{notation*}{Notation}

\theoremstyle{remark}

\newcounter{thecase}
\numberwithin{equation}{section}

\newcommand{\abs}[1]{\lvert#1\rvert}

\begin{document}

\title{Suciu's Conjecture on hyperplane arrangements}

\author[M. Aprodu]{Marian Aprodu}
\address{Marian Aprodu: 
Faculty of Mathematics and Computer Science, 
\hfill \newline\texttt{}
 \indent University of Bucharest,  Romania, and
\hfill \newline\texttt{}
  \indent Simion Stoilow Institute of Mathematics\hfill \newline\texttt{}
 \indent P.O. Box 1-764,
RO-014700 Bucharest, Romania
}
\email{{\tt marian.aprodu@imar.ro}}

\newcommand{\arxiv}[1]
{\texttt{\href{http://arxiv.org/abs/#1}{arXiv:#1}}}
\newcommand{\arx}[1]
{\texttt{\href{http://arxiv.org/abs/#1}{arxiv:}}
	\texttt{\href{http://arxiv.org/abs/#1}{#1}}}
\newcommand{\arxx}[2]
{\texttt{\href{https://arxiv.org/abs/#1.#2}{arxiv:#1.}}
	\texttt{\href{https://arxiv.org/abs/#1.#2}{#2}}}
\newcommand{\doi}[1]
{\texttt{\href{http://dx.doi.org/#1}{doi:#1}}}
\renewcommand{\MR}[1]
{\href{http://www.ams.org/mathscinet-getitem?mr=#1}{MR#1}}
\bibliographystyle{amsplain}

\author[G. Farkas]{Gavril Farkas}
\address{Gavril Farkas: Institut f\"ur Mathematik, Humboldt-Universit\"at zu Berlin \hfill \newline\texttt{}
\indent Unter den Linden 6,
10099 Berlin, Germany}
\email{{\tt farkas@math.hu-berlin.de}}

\author[C. Raicu]{Claudiu Raicu}
\address{Claudiu Raicu: Department of Mathematics,
University of Notre Dame \hfill \newline\texttt{}
\indent 255 Hurley Notre Dame, IN 46556, USA, and \hfill\newline\texttt{}
\indent Simion Stoilow Institute of Mathematics, \hfill\newline\texttt{}
\indent  P.O. Box 1-764, RO-014700 Bucharest, Romania}
\email{{\tt craicu@nd.edu}}

\author[A. Suciu]{Alexander I. Suciu}
\address{Alexander I. Suciu: Department of Mathematics,
Northeastern University \hfill \newline\texttt{}
\indent Boston, MA, 02115, USA,
and
\hfill \newline\texttt{}
  \indent Simion Stoilow Institute of Mathematics\hfill \newline\texttt{}
 \indent P.O. Box 1-764,
RO-014700 Bucharest, Romania}
\email{{\tt a.suciu@northeastern.edu}}

\begin{abstract}
    We prove that the resonance of any hyperplane arrangement is separable. Using the main result from \cite{AFRSarx} this proves Suciu's conjecture on the fundamental group of every hyperplane arrangement, in its most effective form.
\end{abstract}

\maketitle

\section{Introduction}
\label{sect:main}

For a complex hyperplane arrangement $\A$ in $\mathbb C^m$, it is a fundamental question to determine which topological invariants of the complement 
$$M(\A)\coloneqq \mathbb C^m\setminus \bigcup_{H\in \A} H$$
are determined by the intersection lattice (matroid) $L(\A)$ of the arrangement. Building on work of Brieskorn and Arnold, it was shown in \cite{OS} that the cohomology algebra $H^*\bigl(M(\A), \mathbb C\bigr)$ is isomorphic to the \emph{Orlik-Solomon algebra} $A(\A)$, which is an explicit quotient of the exterior algebra $E(\A)$ freely generated in degree one by elements $\{e_H\}_{H\in \A}$. In particular, $H^*(M(\A), \mathbb C)$ is manifestly matroidal. The fundamental group of $M(\A)$
$$G(\A)\coloneqq\pi_1\bigl(M(\A)\bigr)$$ is on the other hand not matroidal \cite{Ryb}, in the sense that there exist arrangements having the same underlying matroid,  whose fundamental groups are however not isomorphic. This fact led to the question of identifying the key attributes of the fundamental group of an arrangement which are matroidal, and to compute these \emph{explicitly} in matroidal terms.

\vskip 4pt

A prime example is provided by the \emph{Chen ranks} of the group $G(\A)$ of the arrangement. Given a finitely generated group $G$, one considers its lower central series $$G=\Gamma_1(G)\supseteq \cdots \supseteq\Gamma_q(G)\supseteq \Gamma_{q+1}(G)\supseteq \cdots ,$$
obtained by taking successive commutators. The metabelian quotient $G_{\mathrm{met}}$ of $G$ is the quotient $G/G''$, where $G'=[G, G]$ and $G''=[G', G']$. If $G$ is the fundamental group of a connected, finite CW-complex $X$, then one can think of $G'/G''$ as the first homology of the universal abelian cover of $X$. The \emph{Chen ranks}  $\theta_q(G)$ are defined as the ranks of the quotients of the lower central series of the metabelian quotient of $G$, that is, 
$$\theta_q(G)\coloneqq\rank \Gamma_q(G_{\mathrm{met}})/\Gamma_{q+1}(G_{\mathrm{met}}).$$
Even in the case of the free group $F_m$, the metabelian quotient encodes a surprising amount of information.  Its Chen ranks have been first computed by K. T. Chen \cite{Ch-ann} and one has $$\theta_q(F_m)=(q-1)\binom{m+q-2}{q}.$$

\vskip 5pt

Based on a large number of examples, Suciu \cite{Su-conm} formulated a precise conjecture concerning the large degree formula for the Chen ranks of any hyperplane arrangement in terms of the associated \emph{resonance variety}. For hyperplane arrangements, this variety is a union of linear spaces which can be described using the beautiful theory of \emph{multinets}. These are, as we shall explain,  matroidal structures associated to any arrangement.  The conjecture has been one of the guiding problems of the field and the Chen invariants of many classes of arrangements have been computed, see  \cites{CSc-adv, CS-tams, DPS-duke, PS, SS-tams02, SS-tams, AFRSarx}. The aim of this paper is to explain how using the main results of \cites{AFRS1, AFRSarx} for arbitrary Koszul modules, that is, those not necessarily attached to an arrangement, we obtain a definitive version of Suciu's Conjecture for all arrangements.

\vskip 4pt

Before formulating Suciu's Conjecture, we recall the definition of the \emph{resonance variety} of an arrangement $\A$, as being the locus
\begin{equation}
\label{eq:def-resonance}
\RR(\A)\coloneqq\Bigl\{a\in H^1\bigl(M(\A), \mathbb C\bigr) : \exists b\in
H^1\bigl(M(\A), \mathbb C\bigr) \mbox{ such that } 0\neq a\wedge b\in \ker(\cup) \Bigr\}\cup \{0\},
\end{equation}
where 
\[
\cup\colon \bwedge^2 H^1\bigl(M(\A),\C\bigr)\longrightarrow H^2\bigl(M(\A), \C\bigr)
\]
is the cup-product map, see also (\ref{eq:defK-from-A}) below. The cohomology $H^*\bigl(M(\A), \mathbb C\bigr)$ being matroidal, the resonance variety is a matroidal invariant of $\A$ as well. The resonance $\RR(\A)$ is a union of linear subspaces which can be explicitly described in terms of multinets on subarrangements of $\A$. Following Falk and Yuzvinsky \cite{FY}, a  $k$-multinet of an arrangement $\A$ corresponds to a partition of $\A$ into $k$ blocks $\A_1$, \ldots, $\A_k$ and an assignment of multiplicities $m\colon \A\rightarrow \mathbb Z_{>0}$ satisfying a number of axioms, see \S\ref{sec:multinets}. We state the main result of this paper:

\begin{theorem}\label{thm:main}
Let $\A$ be a hyperplane arrangement. If $h_m(\mc{A})$ denotes the number of $m$-dimensional components of the resonance of $\mc{A}$, then for  $q\geq \mathrm{max}\{2, |\A|-1\}$, one has 
$$\theta_q\bigl(G(\A)\bigr)=\sum_{m\geq 2} h_m(\A) \cdot \theta_q(F_m)=\sum_{m\geq 2} h_m(\A)\cdot (q-1)\binom{m+q-2}{q}.$$
\end{theorem}

Theorem \ref{thm:main} has been formulated as a conjecture in \cite{Su-conm} for $q\gg 0$. Theorem \ref{thm:main} has been proved in its effective form for all arrangements with only 
double and triple points in \cite[Theorem 9.9]{AFRSarx}. We refer to \cite{AFRSarx} for a detailed history of the problem and for a discussion of previous work done on the subject. 

\vskip 4pt

The fundamental algebraic object behind Theorem \ref{thm:main} is the \emph{Koszul module} of the arrangement $\mc{A}$, whose affine support defines the resonance variety $\RR({\mc{A}})$. These objects have been studied in \cite{AFPRW1, AFRS1, AFRSarx}, building on important earlier work in \cites{PS-imrn, PS-crelle}. Let $V\coloneqq H_1\bigl(M(\A), \C\bigr)$, hence 
$V^{\vee}=H^1\bigl(M(\A), \C\bigr)=\op{Span}{\{e_H\: H\in \A\}}$. We denote by
$S\coloneqq \Sym(V)$ the symmetric algebra on $V$. If 
\begin{equation}\label{eq:defK-from-A}
K\coloneqq\mbox{Im}\Bigl\{\cup\colon \bwedge^2 H^1\bigl(M(\A),\C\bigr)\longrightarrow H^2\bigl(M(\A), \C\bigr)\Bigr\}^{\vee},
\end{equation}
the \defi{Koszul module}\/ $W(\A)$ is 
the graded $S$-module defined as the homology of the complex
\begin{equation}
\label{eq:def-WVK}
\begin{tikzcd}[column sep=22pt]
K \oo S \ar[r, "\delta_2"] 
& V\oo S(1) \ar[r, "\delta_1"] & S(2),
\end{tikzcd}
\end{equation}
where $\delta_2$ is the Koszul differential 
$(u\wedge v)\otimes f \stackrel{\delta_2}\mapsto v\otimes (u\cdot f)-u\otimes (v\cdot f)$, for 
$u, v\in V$ and $f\in S$, while $\delta_1$ is the multiplication map. The fact that that the support of $W(\A)$ is precisely the resonance variety $\RR(\A)$ defined by (\ref{eq:def-resonance}) is established in \cite[Lemma 2.4]{PS-crelle}.

\vskip 4pt

The degree $q$ part 
$W_q(\A)$ of the Koszul module $W(\A)$ is then related to the Chen ranks of $G(\A)$. 
Precisely, it is proved in \cites{PS-imrn, PS-crelle} that
\begin{equation}\label{eq:chen_arr}
\theta_{q+2}(G(\A))=\mbox{dim }  W_q(\A),
\end{equation}
for each $q\geq 0$. This makes the task of determining the Chen ranks of arrangement amenable to algebro-geometric methods. 


\vskip 4pt

Denoting by $I^2(\A)=K^{\perp}\subseteq\bigwedge^2 V^{\vee}$ the degree $2$ part of the Orlik--Solomon ideal of $\A$, a  linear subspace $P\subseteq V^{\vee}$ is said to be \emph{isotropic} if $\bw^2 P\subseteq I^2(\A)$, and \emph{separable} if
\begin{equation}
\label{eq:separable}
\bigl(P\wedge V^{\vee}\bigr)\cap I^2(\A)\subseteq\bwedge^2 P .
\end{equation}
A linear component of the resonance $\RR(\A)$ that is both isotropic and separable is called \emph{strongly
isotropic}; we say that $\RR(\A)$ is isotropic, separable or strongly isotropic if it is linear and each of its components is isotropic, separable or strongly isotropic, respectively. Separability and strong isotropicity were introduced and studied in \cite{AFRS1} in a more general, purely algebraic context. 
 
It follows from \cite{DPS-duke} that the resonance $\RR(\A)$ of any arrangement is a union of linear subspaces, moreover each of these subspaces is isotropic. Furthermore, via \cite{FY} components of $\RR(\A)$ are described by multinets $\M$ on subarrangements $\mathcal{B}$ of $\A$, which makes the resonance $\RR(\A)$ and the numerical invariants $h_m(\A)$ appearing in Theorem \ref{thm:main} \emph{explicitly}  matroidal. We now state the main result of this note:

\begin{theorem}\label{thm:separable}
The resonance variety of each hyperplane arrangement $\A$ is strongly isotropic.
\end{theorem}

We explain a few facts about the proof of Theorem \ref{thm:separable}. In \cite[Proposition 9.4]{AFRSarx}, we identified three types of equations 
$$\Bigl\{\partial=0, \partial_X=0, \ \ \mbox{ for } X\in \mathcal{X} \ \mbox{ and } \ \partial_H=0, \mbox{ for } H\notin \mathcal{B}\Bigr\}$$ for the second exterior power $\bw^2P_{\mathcal{N}}$ of any irreducible component $P_{\mathcal{N}}$ of the resonance induced by a multinet $\mathcal{N}$ on a subarrangement $\mathcal{B}\subseteq \mc{A}$. We then proved in \cite[Theorem 2.5]{AFRSarx} that the first two types of equations also hold true for the subspace $(P\wedge V^\vee)\cap I^2(\A)$. 
The equations of the third type were verified in various special situations, cf. \cite[\S9]{AFRSarx}, and it is precisely their verification for maximal isotropic subspaces that obstructed the case of general arrangements treated here. We are able to avoid these equations, by applying two new ideas. 

\vskip 3pt

The first is the use of centralizers, which were introduced  in \cite{AFRW}.  We extend this notion  to subspaces in Definition~\ref{def:cent}. Maximality becomes then the condition that a subspace coincides with its centralizer. Lemma~\ref{lem:rigidity}, and Proposition~\ref{prop:propagation} yield then sufficient conditions for an element to be contained in the centralizer of a subspace associated to a multinet, revealing a propagation phenomenon: conditions imposed at a single block already force the required behaviour at all of them.  The second idea is that, instead of focusing on separability for maximal isotropic subspaces, we establish a weaker condition involving centralizers and which makes sense for the subspace attached to an arbitrary multinet, see Theorem~\ref{thm:factorization}. For maximal isotropic subspaces, this condition specializes to separability, and every positive-dimensional component of the resonance is such a subspace, by Theorem~\ref{thm:multinet-spaces}.

\medskip

The proof of Suciu's Conjecture (Theorem \ref{thm:main}) now follows by combining \cite[Theorem 1.1]{AFRSarx} and Theorem \ref{thm:separable}. Indeed, the main result of \cite{AFRSarx} provides a formula for the dimensions $W_q(V,K)$ when $q\geq \mbox{dim}(V)-3$, for any Koszul module whose resonance variety $\RR(V,K)$ is strongly isotropic. Coupled with (\ref{eq:chen_arr}) this provides the desired formula for $\theta_q\bigl(G(\A)\bigr)$. Showing that the hypotheses of \cite[Theorem 1.1]{AFRSarx} hold for an arbitrary arrangement is the main goal of this paper. Important previous work on the resonance of an arrangement in connection to Suciu's Conjecture has been done by Cohen--Schenck \cite{CSc-adv}. Partial results on the strong isotropicity of
$\RR(\A)$ under further restrictions have been established in \cite[Proposition 9.8, Theorem 9.9]{AFRSarx}.

\section{Preliminaries}
In this section, we review notions and known results to be used throughout this note.

\subsection{Chen ranks of Koszul modules}
We recall basic facts concerning Koszul modules following \cite{AFRS1} and \cite{AFRSarx}. We fix an $n$-dimensional complex vector space $V$ and set $S:=\mbox{Sym}(V)$. For a linear subspace $K\subseteq \bigwedge^2 V$, we form the graded $S$-module $W(V,K)$, whose degree $q$ component equals
\[
W_q(V,K)=\mathrm{homology}\Bigl\{\!\!
\begin{tikzcd}[column sep=22pt]
 K\otimes \Sym^q V\ar[r,"\delta_2"] &
 V\otimes \Sym^{q+1} V\ar[r,"\delta_1"] & \Sym^{q+2} V   
\end{tikzcd}
\!\!\Bigr\}.
\]
Here the differentials are given by $\delta_2\bigl((u\wedge v)\otimes f\bigr):=v\otimes (uf)-u\otimes (vf)$, whereas $\delta_1(u\otimes f):=uf$ is the multiplication map.

\vskip 4pt

Associated to $(V, K)$, we also have the \emph{resonance variety} defined as the locus 
\begin{equation}
\label{eq:def-resonance-bis}
\RR(V,K)\coloneqq\Bigl\{a\in V^\vee : \text{ there exists $b\in
V^\vee$ such that $a\wedge b\in K^\perp\setminus \{0\}$} \Bigr\}\cup \{0\}.
\end{equation}
It is shown in \cite{PS-crelle} that $\mbox{supp } W(V,K)=\RR(V,K)$. Koszul modules with vanishing resonance are treated in \cites{AFPRW1, AFPRW2}, whereas Koszul modules with linear resonance are discussed in \cite{AFRSarx}. We are interested in Koszul module whose resonance variety is a union of linear subspaces 
$$\RR(V,K)=\overline{V}^{\vee}_1\cup \cdots \cup \overline{V}^{\vee}_k,$$
where $\overline{V}_i^{\vee}$ is a linear subspace of $V^{\vee}$ corresponding to a surjection $V\twoheadrightarrow \overline{V}_i$. If, in addition,  $\RR(V,K)$ is strongly isotropic, we show in \cite[Theorem 1.1]{AFRSarx}  the following formula for the dimension of the graded pieces of $W(V,K)$:
\begin{equation}\label{eq:chenrk_gen}
\mbox{dim } W_q(V,K)=\sum_{i=1}^k \mbox{dim } W_q(\overline{V}_i, 0), \ \ \mbox{ for all } q\geq n-3.
\end{equation}
Throughout our paper we will be concerned with the case when $\A$ is a hyperplane arrangement,  $V=H_1\bigl(M(\A), \C\bigr)$ and $K$ is defined by \eqref{eq:defK-from-A}. We write $W(\A):=W(V,K)$. We shall explain how Theorem~\ref{thm:main} (that is, Suciu's Conjecture in its effective form), can be understood as a special case of formula \eqref{eq:chenrk_gen} established in \cite{AFRSarx}.

\subsection{Calculations in  the Orlik--Solomon algebra}
\label{sec:rank-two}
Given a hyperplane arrangement $\A$, we denote by $L(\A)$ the intersection lattice of all flats $\bigcap_{H\in \BB}H$, for some subarrangement $\BB$ of $\A$. Recall that the rank of a flat is its codimension. For a rank-two flat $X\in L_2(\A)$ we define 
\[\A_X:=\bigl\{H\in \A: X\subseteq H\bigr\}, \text{ and }
V^{\vee}_X:=\spn\bigl\{e_H: H\in \A_X\bigr\}\subseteq H^1\bigl(M(\A), \C\bigr)=V^{\vee}.\] 
For a vector $v=\sum_H v_H\cdot  e_H\in V^{\vee}$, we write
\[
v_X := \sum_{H\in\A_X} v_H\cdot e_H \quad\text{and}\quad \partial_X v:=\sum_{H\in \A_X} v_H .
\]
Following \cite[\S 9]{AFRSarx}, we use the same symbol $\partial_X$ for the degree $-1$ derivation of $\bw V^{\vee}_X$
determined by $\partial_X(e_H)=1$, for all $H\in \A_X$.  We write $\partial$ for the derivation of
$\bw V^{\vee}$ with $\partial(e_H)=1$ for all $H\in \A$.
For a flat $X\in L_2(\A)$, we consider the hyperplane $P_X\subseteq V_X^{\vee}$ defined by
\begin{equation}\label{eq:defPX}
    P_X = \ker\bigl\{\partial_X:V_X^{\vee}\lra\bb{C}\bigr\} = \op{Span}\left\{e_H-e_{H'}:H,H'\in\A_X \right\}
\end{equation}

\noindent There is a direct sum decomposition $\bw^2 V^{\vee} = \bigoplus_{X\in L_2(\A)} \bw^2 V_X^{\vee}$ with respect to which the quadratic part of the Orlik-Solomon ideal decomposes as follows, see also \cite{OS}
\begin{equation}
\label{eq:brieskorn}
I^2(\A)=\bigoplus_{X\in L_2(\A)} I^2(\A_X), \qquad
I^2(\A_X)=\bw^2 P_X=\ker\bigl\{\partial_X\colon \bwedge^2 V^{\vee}_X\longrightarrow  V^{\vee}_X\bigr\},
\end{equation}
where $I^2(\A_X)$ is non-zero precisely when $\abs{\A_X}\geq 3$. In particular, we have the following equivalence
\begin{equation}\label{eq:local-to-global}
a\wedge b\in I^2(\A) \text{ if and only if }a_X \wedge b_X \in I^2(\A_X), \text{ for all }X\in L_2(\A).
\end{equation}
The Leibniz rule and \eqref{eq:brieskorn} imply the following, see \cite[Lemmas 3.1--3.3]{LY00}.

\begin{lemma}
\label{lem:local-wedge}
Let $a,b\in V^{\vee}$ and $X\in L_2(\A)$. Then
\begin{equation}
\label{eq:wedge-criterion}
a_X\wedge b_X\in I^2(\A_X)
\iff
(\partial_X a)\, b_X=(\partial_X b)\, a_X.
\end{equation}
In particular:
\begin{enumerate}[label=\textup{(\roman*)},itemsep=2pt]
\item \label{lw1} 
If $\partial_X a\ne 0$, then $a_X\wedge b_X\in I^2(\A_X)$ if and only
if $b_X$ is proportional to $a_X$.
\item \label{lw2}
If $a_X\ne 0$ and $\partial_X a=0$, then
$a_X\wedge b_X\in I^2(\A_X)$ if and only if $\partial_X b=0$.
\end{enumerate}
\end{lemma}

\subsection{Isotropy and centralizers}
\label{subsec:iso-Aomoto}

Following \cite[p. 10]{AFRW}, we define the \emph{centralizer} of an element $a\in V^{\vee}$, by setting
\[
\Cent_{\A}(a)=:\bigl\{b\in V^{\vee}: a\wedge b\in I^2(\A)\bigr\}.
\] 

Using this notion, a subspace $P\subseteq V^{\vee}$ is isotropic if and only if $P\subseteq \Cent_{\A}(p)$ for
every $p\in P$.
Moreover, $a\in \RR(\A)$ precisely when $\dim \Cent_{\A}(a)\ge 2$, and therefore any isotropic subspace of dimension at least two must be contained in $\RR(\A)$. This motivates the following.

\begin{definition}
\label{def:cent}
The \emph{centralizer} of a subspace $P\subseteq V^{\vee}$ is defined as 
\[
\Cent_{\A}(P)\coloneqq \bigcap_{p\in P} \Cent_{\A}(p)
=\bigl\{g\in V^{\vee}: p\wedge g\in I^2(\A)\ \text{for all}\ p\in P\bigr\}.
\]
\end{definition}

We have the following elementary lemma concerning maximal isotropic subspaces of $V^{\vee}$:

\begin{lemma}
\label{lem:rigidity}
Let $P\subseteq V^{\vee}$ be an isotropic subspace with $\dim P\ge 2$. Then $P$ is maximal among the isotropic subspaces of $V^{\vee}$ if and only if $\Cent_{\A}(P)=P$. 
\end{lemma}

\begin{proof}
Assume first that $P$ is maximal among the isotropic subspaces of $V^{\vee}$.
Since $P$ is isotropic, we have $P\subseteq \Cent_{\A}(P)$. Suppose there exists $g\in \Cent_{\A}(P)\setminus P$. If we consider $P':=P+\mathbb C\cdot g$, then for all $p,p'\in P$ and all $\lambda,\mu\in \C$, we obtain
\[
(p+\lambda g)\wedge (p'+\mu g)=p\wedge p'+\mu\, p\wedge g-\lambda\, p'\wedge g
\in I^2(\A).
\]
This implies that $P'$ is isotropic, contradicting the maximality of $P$.

Conversely, assume $\Cent_{\A}(P)=P$ and let $P'\supseteq P$ be isotropic. For every $g\in P'$ and every $p\in P\subseteq P'$ we have $p\wedge g\in\bw^2P'\subseteq I^2(\A)$, so $g\in\Cent_{\A}(P)=P$, and hence $P'=P$ and $P$ is maximal.
\end{proof}


\subsection{Multinets and resonance components}
\label{sec:multinets}

Following Falk and Yuzvinsky \cite{FY}, we recall next the notion of multinet and record some important properties.

\begin{definition}
\label{def:multinet}
Let $\BB\subseteq \A$ be a subarrangement. A \emph{$k$-multinet}
$\NN=(\BB_1,\dots,\BB_k;m;\XX)$ on $\BB$, with $k\ge 3$, consists of a partition
$\BB=\BB_1\sqcup \cdots \sqcup \BB_k$, an assignment $m\colon \BB\rightarrow \mathbb Z_{>0}$, and a
subset $\XX\subseteq L_2(\BB)$, called the \emph{base locus} of the multinet, such that:
\begin{enumerate}[itemsep=2pt]
\item \label{m1} There is an integer $d>0$ with $\sum_{H\in \BB_i}m_H=d$, for
$i=1,\dots,k$.
\item \label{m2} If $H\in \BB_i$ and $L\in \BB_j$ with $i\ne j$, then $H\cap L\in \XX$.
\item \label{m3} For each $X\in \XX$, the integer
$n_X\coloneqq \sum_{H\in \BB_i,\, X\subseteq H} m_H$ does not depend on $i$.
\item \label{m4} For $i=1, \ldots, k$, the graph $\Gamma_i$ having vertex set $\BB_i$ and an edge
$\{H,L\}$ whenever $H\cap L\notin \XX$ is connected.
\end{enumerate}

\vskip 5pt

The integer $d$ is the \emph{weighted degree} of $\NN$. Set
$u_i\coloneqq \sum_{H\in \BB_i}m_H\cdot e_H\in V^{\vee}$ and
\begin{equation}
\label{eq:PN}
P_{\NN}\coloneqq \Bigl\{\sum_{i=1}^k \lambda_i u_i : \sum_{i=1}^k \lambda_i=0\Bigr\}
=\spn\{u_2-u_1,\dots,u_k-u_1\}\subseteq V^{\vee} .
\end{equation}
The vectors $u_i$ are linearly independent, and therefore $\dim P_{\NN}=k-1\ge 2$.
\end{definition}

We record some important consequences of Definition~\ref{def:multinet}. By \cite[Lemma~2.3(iii)]{FY}, we have
\begin{equation}\label{eq:deg1}
\sum_{X\in \XX,\, X\subseteq H} n_X=d.
\end{equation}
As noted in \cite[p.~1073]{FY}, the base locus $\mc{X}$ is determined by the partition $\BB=\BB_1\sqcup \cdots \sqcup \BB_k$:
\begin{equation}
\label{eq:X-from-B_i}
    \mc{X}=\bigl\{H\cap L : H\in\BB_i,\ L\in\BB_j,\ i\neq j\bigr\}.
\end{equation}
In particular $n_X>0$ for $X\in \mc{X}$, hence each $\BB_i$ contains some hyperplane passing through~$X$. Finally, if $r_H\coloneqq \abs{\{X\in \XX: X\subseteq H\}}$ then in the proof of \cite[Lemma~3.10]{FY} it is shown that
\begin{equation}\label{eq:deg2}
\sum_{\substack{L\in \BB_i\setminus\{H\}\\ H\cap L\notin \XX}} m_L=(r_H-1)\, m_H .
\end{equation}

The simplest examples of multinets recover the subspaces of the form $P_X$ in \eqref{eq:defPX} which have dimension at least $2$. For a flat $X\in L_2(\A)$ with $|\A_X|=k\geq 3$, let $\mc{B}$ contain each of the hyperplanes in $\A_X$ and set all multiplicities equal to $1$. We obtain a $k$-multinet $\mc{N}$ of weighted degree $d=1$ and base locus $\XX=\{X\}$, and $P_{\mc{N}}=P_X$.

\vskip 3pt

We summarize next some important consequences of \cites{DPS-duke,FY,LY00} in the form they will be needed in the rest of this paper.

\begin{theorem}\label{thm:multinet-spaces} Notation as above.

\noindent (1) Each space $P_{\NN}$ in \eqref{eq:PN} is an isotropic subspace of $V^\vee$ contained in $\mc{R}(\mc{A})$.

\noindent (2) Every positive-dimensional irreducible component $P$ of $\RR(\A)$ is a linear subspace, and $P=P_{\NN}$ for some multinet $\NN$ on a subarrangement. 
Moreover, every positive-dimensional irreducible component of $\RR(\A)$ is maximal among the isotropic subspaces of $V^{\vee}$.
\end{theorem}

The components of type $P_X$ discussed above are called \emph{local}, and the components supported on the whole arrangement (that is, corresponding to multinets on $\mc{A}$) are called \emph{essential}.

\section{A generalization of separability for multinet subspaces}

Throughout this section $\NN=(\BB_1,\dots,\BB_k;m;\XX)$ is a $k$--multinet on a subarrangement $\BB\subseteq\A$, and $P=P_{\NN}\subseteq V^{\vee}$ is the associated subspace \eqref{eq:PN} of the resonance $\RR(\A)$. In order to prove Theorem \ref{thm:separable}, we need to verify the separability of resonance components. We do so by establishing a more general statement without assuming that $P=P_{\NN}$ is a component.

\begin{theorem}
\label{thm:factorization}
Let $\NN$ be a multinet on a subarrangement $\BB\subseteq \A$ and set $P=P_{\NN}$. Then
\begin{equation}
\label{eq:factorization}
\bigl(P\wedge V^{\vee}\bigr)\cap I^2(\A)=P\wedge \Cent_{\A}(P).
\end{equation}
\end{theorem}

The first step in the proof is to give a  characterization in \S\ref{sec:cent} for the containment $g\in \Cent_{\A}(P)$ via conditions imposed by rank-two flats on the coefficients of $g$ in each of the blocks $\mc{B}_1, \ldots, \mc{B}_k$. The main technical work takes place in \S\ref{sec:propagation}, where we show how such~conditions, imposed at only one block and combined with $\partial g=0$, propagate to all blocks. Once this is established, the proof of Theorem~\ref{thm:factorization} follows by a direct verification of the required conditions. The proof is given in \S\ref{sec:factorization}, where we also explain how Theorems~\ref{thm:main} and \ref{thm:separable} follow as consequences.

\subsection{A  characterization of the centralizer of a multinet subspace}
\label{sec:cent}

\begin{proposition}
\label{prop:cent-local}
Let $P=P_{\NN}$ be as above and $g\in V^{\vee}$. Then $g\in \Cent_{\A}(P)$ if and only~if:
\begin{enumerate}[label=\textup{(C\arabic*)},itemsep=2pt]
\item \label{C1} $\partial_X g=0$, for every $X\in \XX$, and
\item \label{C2} for every block $\BB_i$ and every $Y\in L_2(\A)\setminus \XX$ meeting $\BB_i$, $g_Y$ is proportional to $(u_i)_Y$.
\end{enumerate}
\end{proposition}

\begin{proof} We write a general element $p\in P$ as $p=\sum_i \lambda_i u_i\in P$, with $\sum_i \lambda_i=0$. In particular,
\begin{equation}\label{eq:parXp=0}
 \partial_X(p) = \sum_i \lambda_i \partial_X(u_i)=n_X \sum_i \lambda_i = 0\text{ for every }X\in \XX.
\end{equation}
Suppose first that $g\in \Cent_{\A}(P)$.  If $X\in\XX$ then using \eqref{eq:X-from-B_i} we can find block indices $i\neq j$ such that $X$ meets $\BB_i$ and $\BB_j$, so $(u_i)_X\neq 0\neq (u_j)_X$. Taking $p=u_i-u_j$ we have $p_X\neq 0$ and $\partial_X(p)=0$, while \eqref{eq:local-to-global} implies that $p_X\wedge g_X\in I^2(\A_X)$. Applying Lemma~\ref{lem:local-wedge}\ref{lw2}, we conclude that $g$ satisfies \ref{C1}. If $Y\in L_2(\A)\setminus \XX$ meets $\BB_i$ then $(u_i)_Y\neq 0$ and using Definition~\ref{def:multinet}\eqref{m2}, $Y$ meets no other block. Taking $p=u_i-u_j$ for some $j\neq i$, we get $p_Y=(u_i)_Y\neq 0$, and $\partial_Y (p)=\partial_Y (u_i)=\sum_{H\in \BB_i,\, Y\subseteq H} m_H>0$. Applying \eqref{eq:local-to-global} again implies $p_Y\wedge g_Y\in I^2(\A_Y)$, and therefore Lemma~\ref{lem:local-wedge}\ref{lw1} shows that $g$ satisfies \ref{C2}.

\vskip 4pt

For the reverse implication, if we assume $g$ satisfies \ref{C1} and \ref{C2}, it suffices by \eqref{eq:local-to-global} to check that $p_Y\wedge g_Y\in I^2(\A_Y)$ for all $p\in P$ and $Y\in L_2(\A)$. If $X\in\XX$ then \eqref{eq:parXp=0}, \ref{C1}, and \eqref{eq:wedge-criterion} imply $p_X\wedge g_X\in I^2(\A_X)$. If $Y\in L_2(\A)\setminus \XX$ meets some block $\BB_i$ then this block is unique, hence either $p_Y=\ll_i(u_i)_Y$ or $g_Y$ is proportional to $p_Y$ by \ref{C2}, proving $p_Y\wedge g_Y=0$. If $Y$ meets no $\BB_i$ then $p_Y=0$ and we have again $p_Y\wedge g_Y=0$, concluding the proof.
\end{proof}

Condition \ref{C2} for a fixed block $\BB_i$ of the multinet $\mathcal{N}$ has the following important consequence regarding incidence conditions for hyperplanes outside $\BB$.

\begin{lemma}
\label{lem:outside}
Fix a block index $i$ and assume that $g\in V^{\vee}$ is such that $g_Y$ is proportional to
$(u_i)_Y$ for every $Y\in L_2(\A)\setminus \XX$ meeting $\BB_i$. If $H\in \A\setminus
\BB$ satisfies $g_H\ne 0$, then
\begin{equation}
\label{eq:allbase}
H\cap L\in \XX \qquad \text{for every } L\in \BB.
\end{equation}
\end{lemma}

\begin{proof} To prove \eqref{eq:allbase} for $L\in\BB_i$, we suppose by contradiction that $Y=H\cap L\not\in\XX$. It follows that $g_Y=c\,(u_i)_Y$ for some scalar $c$, and since $g_H\neq 0$ we get $(u_i)_H\neq 0$. This forces $H\in\BB_i$, contradicting the assumption $H\in\A\setminus\BB$.

Consider now a general block index $j$. A hyperplane $L\in\BB_j$ contains at most one flat $X\in\XX$ with $X\subset H$ (otherwise we would have $L=H$, but $H\not\in\BB$) and therefore
\begin{equation}\label{eq:double-sum-XinXXandH}
\sum_{X\in \XX,\, X\subset H} n_X
=\sum_{\substack{L\in \BB_j\\ L\cap H\in \XX}} m_L
\le \sum_{L\in \BB_j} m_L=d .
\end{equation}
Moreover, equality holds in \eqref{eq:double-sum-XinXXandH} for $j=i$ because the condition $L\cap H\in \XX$ holds for $L\in\BB_i$. Since the two sides of \eqref{eq:double-sum-XinXXandH} are independent of $j$, equality must in fact hold for all $j$. Since $m_L>0$ for $L\in\BB_j$, we conclude that $L\cap H\in\XX$ for all $L\in\BB_j$, as desired.
\end{proof}

\subsection{The weighted Laplacian of a block}
\label{sec:propagation}

Following ideas which already appeared in \cites{FY,LY00}, we define for each block $\BB_j$ of $\mathcal{N}$ the $\abs{\XX}\times \abs{\BB_j}$ incidence matrix $\bl{J}_j$ where $(\bl{J}_j)_{X,H}=1$ if $X\subseteq H$ and $(\bl{J}_j)_{X,H}=0$ otherwise. We let $\bl{E}_j$ be the all-ones square
matrix of size $\abs{\BB_j}$, and set
\begin{equation}
\label{eq:Qj}
\bl{Q}_j\coloneqq \bl{J}_j^{T}\bl{J}_j-\bl{E}_j, \qquad
\bl{D}_j\coloneqq \operatorname{diag}(m_H)_{H\in \BB_j}, \qquad
m_j\coloneqq (m_H)_{H\in \BB_j}\in \mathbb C^{|\mathcal{B}_j|} .
\end{equation}
With the notation in Section~\ref{sec:multinets}, we have \[(\bl{Q}_j)_{H,H}=r_H-1, \text{ and for }L\neq H\in\BB_j,\ (\bl{Q}_j)_{L,H} = \begin{cases}
  0 & \text{if }L\cap H\in\XX, \\
  -1 & \text{if }L\cap H\not\in\XX.
\end{cases}
\]
The matrix $\bl{Q}_j$ is closely related to a weighted Laplacian associated to the graph $\Gamma_j$ from Definition~\ref{def:multinet}\eqref{m4}, as follows. Recall that $\Gamma_j$ has vertex set $\BB_j$, and an edge $\{H,L\}$ whenever $H\cap L\notin \XX$, and we assign each edge $\{H,L\}$ weight $m_Hm_L>0$. 

\vskip 3pt

The \emph{weighted Laplacian} $\bl{L}_j$ is then defined by
\[(\bl{L}_j)_{H,H} = \sum_{L\cap H\not\in\XX} m_L\cdot m_H \overset{\eqref{eq:deg2}}{=}(r_H-1)m_H^2,\text{ and  } (\bl{L}_j)_{L,H} = \begin{cases}
  0 & \text{if }L\cap H\in\XX, \\
  -m_Hm_L & \text{if }L\cap H\not\in\XX.
\end{cases} \]
Since $\Gamma_j$ is connected, the kernel is one-dimensional, spanned by the all-ones vector $\bl{1}_j\in\bb{C}^{\BB_j}$, see also  \cite{chung}*{Chapter~1}. Using \eqref{eq:Qj}, it follows that we can write $\bl{L}_j=\bl{D}_j\bl{Q}_j\bl{D}_j$, and therefore
\begin{equation}
\label{eq:kerQ}
\ker \bl{Q}_j=\C\,(\bl{D}_j\bl{1}_j)=\C\,\cdot m_j .
\end{equation}

\begin{lemma}
\label{lem:blockwise}
Let $g\in V^{\vee}$ such that $\partial g=0$ and $\partial_X g=0$ for all $X\in \XX$. We
assume that every $H\in \A\setminus \BB$ with $g_H\ne 0$ satisfies \eqref{eq:allbase}.
Then for each $j$ there is a scalar $c_j$ with
\begin{equation}
\label{eq:proportional}
g_H=c_j\, m_H, \qquad \text{ for all }H\in \BB_j.
\end{equation}
\end{lemma}

\begin{proof}
Set $\mc{O}:=\{H\in \A\setminus \BB: g_H\ne 0\}$, then let $g_j\in \C^{\BB_j}$ and $g_{\mc{O}}\in \C^{\mc{O}}$ be the corresponding coordinate columns of $g$, and let $\bl{M}$ be the $\abs{\XX}\times \abs{\mc{O}}$ incidence matrix defined by $M_{X,H}=1$ if $X\subseteq H$ and $M_{X,H}=0$ otherwise. The equations $\partial_X g=0$, for $X\in \XX$, assemble into the identity
\begin{equation}
\label{eq:incidence-system}
\bl{J}_1g_1+\cdots+\bl{J}_kg_k+\bl{M}g_{\mc{O}}=0  \in \C^{\XX}.
\end{equation}
Fix a block index $j$. If $h\ne j$, then any $H\in \BB_j$ and $L\in \BB_h$ determine a unique base-point $H\cap L\in\XX$, by Definition~\ref{def:multinet}\eqref{m2}. It follows that $\bl{J}_j^{T}\bl{J}_h=\bl{E}_{jh}$ is the $\abs{\BB_j}\times \abs{\BB_h}$ matrix, having all entries equal to $1$. Likewise,
the condition \eqref{eq:allbase} implies $\bl{J}_j^{T}\bl{M}=\bl{E}_{j\mc{O}}$, this being the $\abs{\BB_j}\times \abs{\mc{O}}$ matrix having all entries equal to $1$. Writing
$s_h=\sum_{L\in \BB_h}g_L$ and $s_{\mc{O}}=\sum_{H\in \mc{O}}g_H$ and multiplying \eqref{eq:incidence-system} by $\bl{J}_j^{T}$, we obtain
\[
\bl{J}_j^{T}\bl{J}_jg_j+\Bigl(\sum_{h\ne j}s_h+s_O\Bigr)\mathbf{1}_j=0 .
\]
The equation $\partial g=0$ implies $\sum_h s_h+s_{\mc{O}}=0$, so $\sum_{h\ne j}s_h+s_O=-s_j$.
Since $\bl{E}_jg_j=s_j\mathbf{1}_j$, it follows that $\bl{Q}_jg_j=\bigl(\bl{J}_j^{T}\bl{J}_j-\bl{E}_j\bigr)g_j=0$, and \eqref{eq:kerQ} implies $g_j\in \C m_j$, which yields \eqref{eq:proportional}.
\end{proof}

We provide a sufficient condition for an element to belong to the centralizer of a resonance component:

\begin{proposition}
\label{prop:propagation}
Fix a block index $i$ and assume that an element $g\in V^{\vee}$ satisfies:
\begin{enumerate}[label=\textup{(P\arabic*)},itemsep=2pt]
\item \label{P1} $\partial g=0$;
\item \label{P2} $\partial_X g=0$, for every $X\in \XX$;
\item \label{P3} $g_Y$ is proportional to $(u_i)_Y$ for every
$Y\in L_2(\A)\setminus \XX$ meeting $\BB_i$.
\end{enumerate}
Then $g\in \Cent_{\A}(P_{\NN})$.
\end{proposition}

\begin{proof} In order to prove that $g\in \Cent_{\A}(P_{\NN})$, we verify that $g$ satisfies the conditions from Proposition \ref{prop:cent-local}. Note that condition \ref{C1} is the same as \ref{P2}, so it suffices to verify \ref{C2}.

\vskip 4pt

Combining \ref{P3} and Lemma \ref{lem:outside}, it follows that every $H\in \A\setminus \BB$ with $g_H\ne 0$
satisfies \eqref{eq:allbase}. Using \ref{P1}, \ref{P2} and Lemma~\ref{lem:blockwise}, we can find scalars $c_1,\cdots,c_k$ such that \eqref{eq:proportional} holds.

\vskip 4pt

To check \ref{C2}, we consider $Y\in L_2(\A)\setminus\XX$ meeting a block $\BB_j$ and proceed to proving that $g_Y = c_j(u_j)_Y$. Since $g_H=c_j m_H$ for $H\in\BB_j$, it suffices to prove that $g_H=0$ when $Y\subset H$ and $H\not\in\BB_j$. Suppose that there exists such $H$ with $g_H\neq 0$. Since $Y$ meets no block of $\BB$ different from $\BB_j$ (by \eqref{eq:X-from-B_i}), we can further assume that $H\in
\A\setminus\BB$. Since $Y$ meets $\BB_j$, there exists $L\in\BB_j$ such that $Y\subset L$, and \eqref{eq:allbase} implies that $Y=H\cap L\in\XX$, a contradiction.
\end{proof}

\subsection{The separability of the resonance of an arrangement}
\label{sec:factorization}

With all the ingredients in place, we can now verify Theorem~\ref{thm:factorization} and derive the separability of resonance components for any arrangement.

\begin{proof}[Proof of Theorem~\ref{thm:factorization}]
The inclusion $\supseteq$ follows from Definition \ref{def:cent}, so it suffices to verify the inclusion $\subseteq$. We let $\omega\in \bigl(P\wedge V^{\vee}\bigr)\cap I^2(\A)$ and write, using the basis
\eqref{eq:PN},
\[
\omega=\sum_{i=2}^{k}(u_i-u_1)\wedge h_i , \qquad h_i\in V^{\vee}.
\]
Set $g_i\coloneqq h_i$ for $i\ge 2$ and $g_1\coloneqq -\sum_{i=2}^k h_i$, to obtain the more symmetric representation
\begin{equation}
\label{eq:omega-normal}
\omega=\sum_{i=1}^{k} u_i\wedge g_i\qquad\text{where}\qquad \sum_{i=1}^{k} g_i=0.
\end{equation}

With this notation, we verify the hypotheses \ref{P1}--\ref{P3} of Proposition \ref{prop:propagation} for $g=g_i$, for each block index $i=1, \ldots, k$.

\medskip

\noindent\emph{Condition \ref{P1}.} Since $\omega\in I^2(\A)\subseteq \ker \partial$ and
$\partial u_i=d$ for every $i$, applying $\partial$ to \eqref{eq:omega-normal} yields
\[
0=\partial \omega=\sum_{i=1}^k\bigl(d\, g_i-(\partial g_i)u_i\bigr)=d\sum_{i=1}^k g_i -\sum_{i=1}^k (\partial g_i)\, u_i
\overset{\eqref{eq:omega-normal}}{=}-\sum_{i=1}^k (\partial g_i)\, u_i.
\]
Since the vectors $u_i$ are linearly independent, it follows that $\partial g_i=0$ for all $i$.

\medskip

\noindent\emph{Condition \ref{P2}.} Decomposing $\omega$ with respect to \eqref{eq:brieskorn}, we get $\omega_X=\sum_i (u_i)_X\wedge (g_i)_X\in I^2(\A_X)$, and $\partial_X \omega_X=0$. If $X\in \XX$ then $\partial_X u_i=n_X$ and therefore
\[
0=\partial_X\omega_X
=n_X\sum_{i=1}^k (g_i)_X-\sum_{i=1}^k (\partial_X g_i)\,(u_i)_X \overset{\eqref{eq:omega-normal}}{=} -\sum_{i=1}^k (\partial_X g_i)\,(u_i)_X.
\]
The vectors $(u_i)_X$ are non-zero with pairwise
disjoint supports, hence $\partial_X g_i=0$ for all $i$.

\medskip

\noindent\emph{Condition \ref{P3}.} Suppose that
$Y\in L_2(\A)\setminus \XX$ meets the block $\BB_i$. By Definition~\ref{def:multinet}\eqref{m2} no other block can meet $Y$, so
$(u_l)_Y=0$ for $l\ne i$ and therefore
$\omega_Y=(u_i)_Y\wedge (g_i)_Y\in I^2(\A_Y)$. Since
$\partial_Y (u_i)=\sum_{H\in \BB_i,\, Y\subseteq H}m_H>0$,
Lemma \ref{lem:local-wedge}\ref{lw1} implies that $(g_i)_Y$ is proportional to $(u_i)_Y$.

\medskip

We can therefore apply Proposition \ref{prop:propagation} to conclude that $g_i\in \Cent_{\A}(P)$ for every $i$. In
particular $h_i=g_i\in \Cent_{\A}(P)$ for $i\ge 2$, hence
$\omega=\sum_{i\ge 2}(u_i-u_1)\wedge h_i\in P\wedge \Cent_{\A}(P)$.
\end{proof}

\begin{corollary}
\label{cor:separable-maximal}
If $\NN$ is a multinet on a subarrangement $\BB\subseteq \A$ such that $P_{\NN}$ is maximal among
the isotropic subspaces of $V^{\vee}$, then $P_{\NN}$ is separable.
\end{corollary}

\begin{proof}
By Theorem~\ref{thm:multinet-spaces}, $P_{\NN}$ is isotropic of dimension $\geq 2$. Lemma \ref{lem:rigidity} and the maximality of $P_{\NN}$ implies that $P_{\NN}=\Cent_{\A}(P_{\NN})$. Separability of $P_{\NN}$ follows then from Theorem~\ref{thm:factorization}.
\end{proof}

\begin{proof}[Proof of Theorem~\ref{thm:separable}] 
Let $P$ denote a positive--dimensional irreducible component of $\RR(\A)$. By Theorem~\ref{thm:multinet-spaces}, $P=P_{\NN}$ as in \eqref{eq:PN} and it is isotropic. The conclusion then follows from Corollary~\ref{cor:separable-maximal}.
\end{proof}

\begin{proof}[Proof of Theorem~\ref{thm:main}]
By Theorem~\ref{thm:separable} and \eqref{eq:chenrk_gen}, $\dim W_q(\A)=
\sum_i \dim W_q(\overline V_i,0)$ for $q\ge \abs{\A}-3$, the sum being taken over the components of $\RR(\A)$. The conclusion follows by grouping the components by dimension and using \eqref{eq:chen_arr}.
\end{proof}


{\small{{\bf{Acknowledgments:}} 
Farkas was supported by the Berlin mathematics research center MATH+, by the Einstein Foundation Berlin and by the
ERC Advanced Grant SYZYGY (No. 834172).
Raicu was supported by the NSF Grant No.~2302341.
Suciu was supported in part by the Simons Foundation Collaboration Grant for
Mathematicians No.~693825 and by the project ``Singularities and Applications'' - 
CF 132/31.07.2023 funded by the European Union - NextGenerationEU - through 
Romania's National Recovery and Resilience Plan.}}

{\small{{\bf{AI usage:}} 
The idea of using the weighted Laplacians of the graphs $\Gamma_i$  was proposed by GPT-5.6 Sol after being prompted by the authors to verify the conditions in 
\cite[Proposition 9.4]{AFRSarx}, which ensure the separability of a resonance component. The paper was entirely written by the
authors, who are fully responsible for its correctness.
}}

\end{document}